\documentclass[12pt,a4paper]{article}
\usepackage{amsmath,latexsym,mathrsfs}
\usepackage{amssymb,amsfonts,amscd,amssymb,amsthm}
\usepackage{graphics,verbatim,url}
\newcommand{\doi}[1]{\url{http://dx.doi.org/#1}}

\newtheorem{theorem}{Theorem}[section]
\newtheorem{lemma}[theorem]{Lemma}

\newtheorem{proposition}[theorem]{Proposition}
\newtheorem{remark}[theorem]{Remark}
\theoremstyle{remark}
\newtheorem{definition}[theorem]{\bf Definition}

\numberwithin{equation}{section}

\def\R{{\mathbb R}}

\newcommand{\mS}{\mathbb S^{n-1}}
\newcommand{\mSt}{\mathbb S^{2}}

\newcommand{\cD}{{\cal D}}
\newcommand{\cI}{{\cal I}}
\newcommand{\cK}{{\cal K}}
\newcommand{\cM}{\mathcal M}
\newcommand{\cN}{{\cal N}}
\newcommand{\cV}{\mathcal V}

\newcommand{\inpro}[2]{\left\langle{#1},{#2}\right\rangle}
\newcommand{\inprod}[2]{\left\langle{#1},{#2}\right\rangle}
\newcommand{\norm}[2]{\|{#1}\|_{#2}}
\newcommand{\snorm}[2]{\left|{#1}\right|_{#2}}

\newcommand{\abs}[1]{\left|#1\right|}
\newcommand{\kerL}{\ker\/L}

\newcommand{\De}{\Delta}

\newcommand{\al}{\alpha}

\newcommand{\vecx}{\boldsymbol{x}}
\newcommand{\vecy}{\boldsymbol{y}}

\newcommand{\vecq}{\boldsymbol{q}}
\newcommand{\vecA}{\boldsymbol{A}}
\newcommand{\vecB}{\boldsymbol{B}}
\newcommand{\vecS}{\boldsymbol{S}}

\newcommand{\vecc}{\boldsymbol{c}}

\newcommand{\vecg}{\boldsymbol{g}}

\newcommand{\spann}{{\rm span }}

\newcommand{\goto}{\rightarrow}
\newcommand{\wtd}{\widetilde}
\newcommand{\wth}{\widehat}

\newcommand{\dt}{\, dt}

\newcommand{\notsubset}{\not\subseteq}

\newcommand{\slinN}{\sum_{\ell=0}^\infty\,\sum_{m=1}^{N(n,\ell)}}
\newcommand{\slinK}{\sum_{\ell\in\cK(L)}\,\sum_{m=1}^{N(n,\ell)}}
\newcommand{\slnotinK}{\sum_{\ell\notin\cK(L)}\,\sum_{m=1}^{N(n,\ell)}}

\newcommand{\phe}{\widehat \phi(\ell)}
\newcommand{\Lhe}{\wth L(\ell)}
\newcommand{\Ylm}{Y_{\ell,m}}

\newcommand{\uhlm}{\widehat{u}_{\ell,m}}
\newcommand{\vhlm}{\widehat{v}_{\ell,m}}

\newcommand{\whlm}{\widehat{w}_{\ell,m}}

\newcommand{\pdot}{pseudodifferential operator }

\newcommand{\Pdost}{Pseudodifferential operators }

\newcommand{\pdest}{pseudodifferential equations }
\newcommand{\pdes}{pseudodifferential equations}

\newcommand{\srbfst}{spherical radial basis functions }

\date{}
\title
{Strongly elliptic pseudodifferential equations 
on the sphere with radial basis functions}
\author{T. D. Pham 
        \thanks{School of Mathematics and Statistics, 
                The University of New South Wales, 
                Sydney 2052, Australia. 
                mailto: 
                {\tt Thanh.Tran@unsw.edu.au
}}
       \thanks{Currently at Hausdorff Center for
               Mathematics and Institute for Numerical 
               Simulation, University of Bonn, Germany.
               mailto:
               {\tt duong.pham@hausdorff-center.uni-bonn.de}
               }
   \and T. Tran $^*$
}
\begin{document}
\maketitle

\begin{abstract}
Spherical radial basis functions are used to define 
approximate solutions to strongly elliptic
pseudodifferential equations on the 
unit sphere. These equations arise from geodesy. 
The approximate solutions are found by the Galerkin and 
collocation 
methods. A salient feature of the paper is a {\em
unified theory}
for error analysis of both approximation methods.

\end{abstract}

\section{Introduction}

\Pdost  have long been used \cite{Hor65,KohNir65} as a modern 
and powerful tool to tackle linear boundary-value problems. Svensson
\cite{Sve83} introduces this approach to geodesists who study 
\cite{FreGerSch98,GraKruSch03} these
 problems on the sphere which is taken as a model of the earth. 
Efficient solutions to \pdest on the sphere become more demanding
when given data are collected by satellites.
In this paper, we study the use of spherical radial basis
functions to find approximate solutions to these
equations.

The use of spherical radial basis functions results in
meshless methods which, over the past decades, become more
and more popular \cite{BabBan03, Mel05, Wen99, Wen05}. These methods are
alternatives to finite-element methods.
Solving pseudodifferential equations on the sphere by using
spherical radial basis functions with the collocation method
has been studied by Morton and Neamtu \cite{MorNea02}. 
Error bounds have later been improved by Morton \cite{Mor01};
see also Morton's PhD dissertation \cite{Mor00}.
From the point of view of application, the collocation method is
easier to implement, in particular when the given data are scattered.
However, it is well-known that collocation methods 
in general elicit a complicated error analysis. 
The crux of the analysis in \cite{Mor01,
MorNea02}
is the transformation of the collocation problem
to a Lagrange interpolation problem.

In this paper, first we solve strongly elliptic
\pdest on the sphere by the Galerkin method. (A precise definition
of these equations is delayed until Section~\ref{sec:Pre}.) Error 
analysis is performed with well-known knowledge 
on the Galerkin method. 
Next, we solve the equations by the collocation method.
A salient feature of the paper is that 
error estimates for collocation methods (as considered
in References \cite{Mor00,Mor01,MorNea02}) are obtained as a 
by-product of the analysis for the Galerkin method. This
{\em unified error analysis}
is thanks to an observation that the collocation 
equation can be viewed as a Galerkin equation, due to the reproducing 
kernel property of the space in use. 
Efforts to perform error analysis for the collocation method based 
on that for the Galerkin method have been made by 
several authors to solve quasilinear parabolic equations~\cite{DouDup73},
pseudodifferential equations on closed curves~\cite{ArnWen83}, and
boundary integral equations~\cite{CosSte87}. These approaches use either a
special
set of
collocation points or the duality inner product.

In an earlier paper \cite{PhaTra08a},
we  analyse a collocation approximation
to negative order strongly elliptic \pdes. The results in the present 
paper are more general for operators of any order, 
negative or positive.
Results for elliptic operators will be presented in another paper.

Our error estimates, as compared to those by Morton and Neamtu
\cite{Mor01,MorNea02}, cover a wider range of Sobolev norms.
Indeed, these authors only provide error estimates in the Sobolev norm
 $\norm{\cdot}{2\al}$, where $2\al$ is the order of the operator. 

A study of preconditioning techniques for the Galerkin
method applied to these equations is
carried out in \cite{TraLeGSloSte10}. In the present
paper, we only discuss error estimates.

The paper is organised as follows. In Section~\ref{sec:Pre}, we 
provide necessary ingredients, and define the operators and the 
problem in consideration. Section~\ref{sec:fin dim sub} is devoted 
to the introduction of \srbfst and the approximation spaces to be
employed. Numerical methods are introduced in Section~\ref{sec:App sol}.
Analysis for the Galerkin and collocation methods is carried out in 
Sections~\ref{sec:Gal} and \ref{sec:Col Met}.
Section~\ref{sec:num exp srbf} is devoted to numerical experiments.

Throughout this paper, $C$, $C_1$ and $C_2$ denote
generic constants which may take different values at 
different occurrences.

\section{Preliminaries}\label{sec:Pre}

\subsection{Sobolev spaces}

Throughout this paper, for $n\ge 3$ we denote by $\mS$ the unit sphere in 
$\R^n$, i.e., $\mS:=\{\vecx \in \R^n: |\vecx|=1\}$ where $\snorm{\cdot}{}$ is
the Euclidean 
norm in $\R^n$. 
A spherical harmonic of order $\ell$ on $\mS$ 
is the restriction to $\mS$ of a 
homogeneous harmonic
polynomial of degree $\ell$ in $\R^n$.
The space of all spherical harmonics of order $\ell$ 
is the eigenspace of the Laplace--Beltrami operator
$\De_{\mathbb S}$ corresponding to the eigenvalue 
$\lambda_\ell = - \ell(\ell +n -2)$. The dimension of this space 
being 
\[
 N(n,0)=1 \quad\text{and}\quad
N(n,\ell)
=
\frac{2\ell +n -2}{\ell}{\ell + n - 3 \choose \ell-1},
\quad
\ell\not= 0,
\]
see e.g. \cite[page 4]{Mul66}, one may choose for it an $L_2(\mS)$-orthonormal
basis
$\{Y_{\ell,m}\}_{m=1}^{N(n,\ell)}$. Note that $N(n,\ell)= O(\ell^{n-2})$.
The collection of all the spherical harmonics
$\Ylm$, $m = 1,\ldots, N(n,\ell)$ and $\ell=0,1,\ldots$, forms an orthonormal
basis for $L_2(\mS)$.

For $s\in\R$, the Sobolev space $H^s$ is defined as
usual by
\[
H^s
:=
\Big{\{}
v\in\cD'(\mS): \slinN (\ell+1)^{2s} |\vhlm|^2 < \infty
\Big{\}},
\]
where $\cD'(\mS)$ is the space of distributions on
$\mS$ and $\vhlm$ are the Fourier coefficients of 
$v$,
\[
\vhlm
=
\int_{\mS}
v(\vecx)\Ylm(\vecx){\rm d}\sigma_{\vecx}.
\] 
Here ${\rm d}\sigma_{\vecx}$ is the element of surface
area.
The space $H^s$ is equipped with the following norm
and inner product: 
\begin{equation}\label{equ:Sob nor}
  \|v\|_{s} := 
       \left(\sum_{\ell=0}^\infty \sum_{m=1}^{N(n,\ell)}
          (\ell+1)^{2s} |\widehat{v}_{\ell,m}|^2 \right)^{1/2}
\end{equation}
and
\[
\inpro{v}{w}_{s}
:=
\sum_{\ell=0}^\infty \sum_{m=1}^{N(n,\ell)}
(\ell+1)^{2s} \widehat{v}_{\ell,m} {\widehat{w}_{\ell,m}}.
\]
We note that the series on the right hand side also converges 
when $v\in H^{s+\sigma}$ and $w\in H^{s-\sigma}$ for any 
$\sigma >0$. Therefore, in the following we use the same 
notation $\inprod{\cdot}{\cdot}_s$ for the duality product
between $H^{s+\sigma}$ and $H^{s-\sigma}$.

When $s=0$ we write $\inpro{\cdot}{\cdot}$ instead of
$\inpro{\cdot}{\cdot}_0$; this is in fact the $L_2$-inner
product.
In the sequel, we will frequently use the Cauchy--Schwarz inequality
\begin{equation}\label{equ:CS s}
|\inpro{v}{w}_{s}|
\le
\norm{v}{s} \norm{w}{s}
\quad\text{for all } v,w \in H^s,\ \text{for all } s \in\R,
\end{equation}
and the following identity which can be easily proved
\begin{equation}\label{equ:neg nor}
\norm{v}{s_1}
=
\sup_{w \in H^{s_2} \atop w \not= 0}
\frac{\inpro{v}{w}_{\frac{s_1+s_2}{2}}}{\norm{w}{s_2}}
\quad\text{for all } v \in H^{s_1},\ \text{for all } s_1, s_2 \in\R.
\end{equation}
Identity~\eqref{equ:neg nor} will be used frequently in the proof
of Proposition~\ref{pro:app pro} with different values of 
$s_1$ and $s_2$. 

\subsection{Pseudodifferential operators}
Let $\{\Lhe\}_{\ell\geq 0}$ be a sequence of real numbers.
A pseudodifferential operator $L$ is a linear operator 
that assigns to any $v\in\cD'(\mS)$ a distribution
\[
Lv
:=
\sum_{\ell=0}^\infty \sum_{m=1}^{N(n,\ell)}
\widehat L(\ell) \widehat v_{\ell,m} Y_{\ell,m}.
\]
The sequence $\{\widehat L(\ell)\}_{\ell\geq 0}$ 
is referred to as the {\em spherical symbol} of $L$. 
Let $\cK(L):= \{\ell: \widehat L(\ell) = 0\}$.
Then
\[
\kerL
=
\spann\{Y_{\ell,m} : \ell\in\cK(L), \, m= 1,\ldots, N(n,\ell)\}.
\]
Denoting $M:={\rm dim} \kerL$, we assume that $0\le M<\infty$.

\begin{definition}\label{def:L}
 A \pdot $L$ 
is said to be {\it strongly elliptic of order $2\al$}
if 
\begin{equation}\label{equ:strongly ellip}
C_1(\ell+1)^{2\alpha}
\leq
\widehat L(\ell)
\leq 
C_2(\ell+1)^{2\alpha}
\quad\text{for all }\ell\notin\cK(L),
\end{equation}
for some positive constants $C_1$ and $C_2$. 
\end{definition}

More general pseudodifferential operators can be defined via
Fourier transforms by using local charts; see e.g.,
\cite{HsiWen08, Pet83}.
It can be easily seen that if $L$ is a \pdot of order
$2\al$ then $L: H^{s+\al}\goto H^{s-\al}$ is bounded
for all $s\in \R$. Examples of strongly elliptic
operators of various orders can be found
in~\cite{Sve83}; see also~\cite{TraLeGSloSte10}.

The problem we are solving in this paper is posed as follows.
\paragraph{Problem A:}
\emph{ Let $L$
be a strongly elliptic pseudodifferential operator 
of order $2\al$.
Given, for some $\sigma \ge 0$,  
\begin{equation}\label{equ:g con}
g \in H^{\sigma-\al}
\quad\text{satisfying}\quad
\wth g_{\ell,m} = 0
\quad \text{for all }\, \ell\in \cK(L),\ 
m = 1,\ldots, N(n,\ell),
\end{equation}
find $u \in H^{\sigma+\al}$ satisfying
\begin{equation}\label{equ:sys equ}
\begin{aligned}
Lu &= g, \\
\inprod{\mu_i}{u} &= \gamma_i,\quad i = 1,\ldots, M,
\end{aligned}
\end{equation}
where  $\gamma_i\in\R$ and $\mu_i\in
 H^{-\sigma-\alpha}$ are given. 
 Here $\inprod{\cdot}{\cdot}$ denotes the duality 
 product between $H^{-\sigma-\al}$ and
 $H^{\sigma+\al}$, which coincides with the 
 $H^0$-inner product when $\mu_i$ and $u$ belong to $H^0$.
}

An explanation for the inclusion of $\sigma$
in \eqref{equ:g con} is in order.
For the Galerkin approximation, the energy space is 
$H^\al$. Thus it suffices to assume \eqref{equ:g con} 
with $\sigma=0$. However, for the collocation
approximation, 
it is required that $g$ be at least continuous. Moreover,
we will reformulate the collocation equation
into a Galerkin equation which requires $g\in H^\tau$ 
for some $\tau>0$ to be specified in
 Section~\ref{sec:Col Met}.
 Therefore, we include the constant $\sigma$
 in \eqref{equ:g con}.
 
Problem A is uniquely solvable under the following 
assumption.
\paragraph{Assumption B:}
\emph{
The functionals  $\mu_1,\ldots,\mu_N$ are assumed to be 
\emph{unisolvent} with respect to $\kerL$, i.e.,
for any $v\in\kerL$ if $\inprod{\mu_i}{v}=0$ for all 
$i=1,\ldots,M$, then $v=0$.
}

The following result is proved in \cite{MorNea02}. We include
the proof here for completeness.
\begin{proposition}\label{the:ext uni}
Under Assumption B, Problem A has a unique solution.
\end{proposition}
\begin{proof}
Since $\kerL$ is a finite-dimensional subspace of
$H^{\sigma+\alpha}$,
we can represent $H^{\sigma+\al}$ as
\[
H^{\sigma+\alpha} 
=
\kerL\oplus (\kerL)_{H^{\sigma+\al}}^\perp,
\]
where $(\kerL)_{H^{\sigma+\al}}^\perp$ is the 
orthogonal complement of $\kerL$
with respect to the $H^{\sigma+\alpha}$-inner product.
Writing the solution $u$ in the form
\begin{equation}\label{equ:u u0 u1}
u
=
u_0 + u_1
\quad\text{where}\quad u_0\in \kerL 
\quad\text{and}\quad u_1\in (\kerL)_{H^{\sigma+\al}}^\perp, 
\end{equation}
and noting that $L|_{(\kerL)^\perp_{H^{\sigma+\al} }}$ 
is injective, we can
define $u_1$ by $u_1=L^{-1}g$ and find $u_0\in \kerL$ by solving
\begin{equation}\label{equ:u0 equ}
\inprod{\mu_i}{u_0} = \gamma_i - \inprod{\mu_i}{u_1},
\quad i=1,\ldots,M.
\end{equation}
Since $u_0\in\kerL$, it can represented as
\[
u_0 = 
\slinK
c_{\ell,m} \Ylm.
\]
Substituting this into \eqref{equ:u0 equ} yields
\begin{equation}\label{equ:u0 Ylm}
\slinK
c_{\ell,m}
\inprod{\mu_i}{\Ylm}
=
\gamma_i - \inprod{\mu_i}{u_1},
\quad
i=1,\ldots,M.
\end{equation}
Recalling that $M=\dim \kerL$, we note that there are $M$ 
unknowns $c_{\ell,m}$. The unisolvency assumption B assures us that equation
\eqref{equ:u0 equ} with zero right-hand side
has a unique solution $u_0 = 0$. Therefore, the matrix arising from 
\eqref{equ:u0 Ylm} is invertible, which in turn implies
unique existence of $c_{\ell,m}$, $m=1,\ldots,N(n,\ell)$ and
$\ell\in\cK(L)$.
The proposition is proved.
\end{proof} 

We define a bilinear form 
$a(\cdot,\cdot): H^{\al+s}\times H^{\al-s}\goto\R$,
for any $s\in\R$,  
by
\begin{equation}\label{equ:bilinear srbf}
a(w,v)
:=
\inprod{Lw}{v}
\quad
\text{for all }
w\in H^{\al+s},\
v\in H^{\al-s}.
\end{equation}
In particular, when $s=\sigma$ we have by noting~\eqref{equ:sys equ}

\begin{equation}\label{equ:Lu1}
a(u_1,v)
=
\inpro{g}{v}
\quad \text{for all } v\in H^{\al-\sigma}.
\end{equation}
In the sequel, for any $x,y\in\R$, $x\simeq y$ means that
there exist positive constants $C_1$ and $C_2$
satisfying
$
C_1 x
\leq
y
\leq
C_2 x
$.
The following simple results are often used
in the next sections.
\begin{lemma}\label{lem:a bil for}
Let $s$ be any real number.
\begin{enumerate}
\item
The bilinear form $a(\cdot,\cdot): H^{\al+s}\times H^{\al-s}
\goto\R$ is bounded, i.e., 
\begin{equation}\label{equ:bounded bil form}
|a(w,v)|
\le
C
\norm{w}{\al+s} \norm{v}{\al-s}
\quad
\text{for all } w\in H^{\al+s},\ v\in H^{\al-s}.
\end{equation}
\item
If $w,v\in H^s$, then
\begin{equation}\label{equ:continuous}
|\inpro{Lw}{v}_{s-\al}|
\le
C \norm{w}{s} \norm{v}{s}.
\end{equation}
\item
Assume that $L$ is strongly elliptic.
If 
$v\in(\kerL)^\perp_{H^s}$, then
\begin{equation}\label{equ:equi norm}
\inpro{Lv}{v}_{s-\al}
\simeq
\norm{v}{s}^2.
\end{equation}
In particular, setting $s = \alpha$ in~\eqref{equ:equi norm}, there holds
$a(v,v)\simeq \norm{v}{\al}^2$ 
for all $v\in (\kerL)^\perp_{H^\al}$.
\end{enumerate}
Here $C$ is a constant independent of $v$ and $w$.
\end{lemma}

\begin{proof}
Let $w\in H^{\al+s}$ and $v\in H^{\al-s}$.
Noting \eqref{equ:strongly ellip} 
and using the Cauchy--Schwarz inequality, we have
\begin{align*}
|a(w,v)|
&
\leq
\slinN
|
\Lhe|
|
\whlm
|
|
\vhlm
|
\leq
C
\slinN
(\ell+1)^{2\al}
|
\whlm
|
|
\vhlm
|
\\
&
=
C
\slinN
(\ell+1)^{\al+s}
|
\whlm
|
(\ell+1)^{\al-s}
|
\vhlm
|
\\
&
\leq
C
\left(
\slinN
(\ell+1)^{2(\al+s)}
|\whlm|^2
\right)^{1/2}
\left(
\slinN
(\ell+1)^{2(\al-s)}
|\vhlm|^2
\right)^{1/2}
\\
&
=
C
\norm{w}{\al+s}
\norm{v}{\al-s}, 
\end{align*}
proving \eqref{equ:bounded bil form}. The proof for 
\eqref{equ:continuous} and \eqref{equ:equi norm} can be done 
similarly, noting the definition 
\eqref{equ:strongly ellip} of strongly elliptic operators,
and noting that $v\in(\kerL)^\perp_{H^s}$ if and only if
$v\in H^s$ and $\wth v_{\ell,m}=0$ 
for all
$\ell\in\cK(L)$ and $m=1,\ldots,N(n,\ell)$.
\end{proof}
 
In the next section, we shall define finite-dimensional 
subspaces in which approximate solutions are sought for.

\section{Approximation subspaces}\label{sec:fin dim sub}
The finite-dimensional subspaces to be used in the
approximation will be defined from spherical radial basis
functions, which in turn are defined from kernels.
\subsection{Positive-definite kernels}\label{subsec:pos def
ker}
A continuous function $\Theta : \mS\times\mS \to \R$ is called
a \emph{positive-definite kernel} on~$\mS$ if it satisfies
\begin{itemize}
\item[(i)]
 $\Theta(\vecx,\vecy)={\Theta(\vecy,\vecx)}$ for all $\vecx,
\vecy\in\mS$,
\item[(ii)]
for any positive integer $N$ and any set of distinct points
$\{\vecy_1,\ldots,\vecy_N\}$ on~$\mS$, the 
$N\times N$ matrix $\vecB$ with entries
$\vecB_{i,j} = \Theta(\vecy_i,\vecy_j)$ is positive-semidefinite.
\end{itemize}
If the matrix $\vecB$ is positive-definite then
$\Theta$ is called a \emph{strictly positive-definite} kernel;
see \cite{Sch42,XuChe92}.

We characterise the kernel $\Theta$ by a {\it shape
function} $\theta$ as follows. Let
$\theta : [-1,1]\to\R$ 
be a univariate function having
 a series expansion in terms of 
Legendre polynomials,
\begin{equation}\label{equ:phi}
\theta(t)
= 
\sum_{\ell=0}^\infty 
\omega_n^{-1} N(n,\ell)\widehat{\theta}(\ell) P_\ell(n;t),
\end{equation}
where $\omega_n$ is the surface area of the sphere $\mS$, 
and $\wth\theta(\ell)$ is the Fourier--Legendre coefficient, 
\[
\widehat{\theta}(\ell) 
= 
\omega_{n- 1} \int_{-1}^{1} \theta(t) 
P_\ell(n;t) (1-t^2)^{(n-3)/2} \dt.
\]
Here, $P_\ell(n;t)$ denotes the degree $\ell$ normalised Legendre polynomial in
$n$ variables so that $P_\ell(n;1) = 1$, 
as described in \cite{Mul66}.
Using this shape function $\theta$, we define
\begin{equation}\label{equ:Theta theta}
\Theta(\vecx,\vecy)
:=
\theta(\vecx\cdot\vecy)
\quad\text{for all }\vecx,\vecy\in\mS,
\end{equation}
where $\vecx\cdot\vecy$ denotes the dot product
between $\vecx$ and $\vecy$. We note that $\vecx\cdot\vecy$ 
is the cosine of the angle between $\vecx$ and $\vecy$,
 which is called the geodesic
distance between the two points. Thus the kernel $\Theta$ is a
zonal kernel.
By using the well-known addition formula for spherical harmonics 
\cite{Mul66},
\begin{equation}\label{equ:add for}
\sum_{m=1}^{N(n,\ell)}\,
Y_{\ell,m}(\vecx)\,
{Y_{\ell,m}(\vecy) }
=
\omega_n^{-1}N(n,\ell)\,
P_\ell(n;\vecx\cdot\vecy)
\quad
\text{for all } \vecx,\vecy\in \mS,
\end{equation}
we can write
\begin{equation}\label{equ:Phi res}
\Theta(\vecx,\vecy) 
= 
\sum_{\ell=0}^\infty 
\sum_{m=1}^{N(n,\ell)} 
\widehat{\theta}(\ell)
Y_{\ell,m}(\vecx) {Y_{\ell,m}(\vecy)}.
\end{equation}

\begin{remark}\label{rem:ker}
{\rm In \cite{CheMenSun03},
a complete characterisation of strictly positive-definite kernels is
established: the kernel $\Theta$ is strictly positive-definite 
if and only if
$\widehat{\theta}(\ell) \ge 0$ for all $\ell\ge 0$, and 
$\widehat{\theta}(\ell) > 0$ for infinitely many even values of
$\ell$ and 
infinitely many odd values of $\ell$; see also \cite{Sch42} and 
\cite{XuChe92}.
}
\end{remark}

In the remainder of this section, we shall define a specific shape function
$\phi$ and a specific kernel $\Phi$ which will be used to define the
approximation subspace.
The notations $\theta$ and $\Theta$ are reserved for future general reference.

\subsection{Spherical radial basis functions}\label{subsec:Gal strong}

We choose a shape function $\phi$ such that there exists $\tau\in\R$ 
satisfying
\begin{equation}\label{equ:con hat phi}
\widehat{\phi}(\ell) \simeq (\ell+1)^{-2\tau}
\quad
\text{for all }
\ell\ge0.
\end{equation} 
The corresponding kernel $\Phi$ defined by \eqref{equ:Theta theta}, i.e.,
$\Phi(\vecx,\vecy) = \phi(\vecx\cdot\vecy)$,
is then strictly positive-definite; 
see Remark~\ref{rem:ker}. 
The native space associated with $\phi$ is defined by
\[
\cN_{\phi}
:=
\Big{\{}v\in \cD'(\mS) : \norm{v}{\phi}^2 = \sum_{\ell=0}^\infty
\sum_{m=1}^{N(n,\ell)} \frac{|\widehat v_{\ell,m}|^2}{\widehat\phi(\ell)}
< \infty \Big{\}}.
\]
This space is equipped with an inner product and a norm defined by
\[
\inpro{v}{w}_{\phi}
=
\sum_{\ell=0}^\infty \sum_{m=1}^{N(n,\ell)} 
\frac{\widehat v_{\ell,m}{\widehat w_{\ell,m}}}{\widehat\phi(\ell)}
\quad\text{and}\quad
\norm{v}{\phi}
=
\left(\sum_{\ell=0}^\infty \sum_{m=1}^{N(n,\ell)} 
\frac{|\widehat v_{\ell,m}|^2}{\widehat\phi(\ell)}\right)^{1/2}.
\]
Since $\phe$ satisfies \eqref{equ:con hat phi}, the native space
$\cN_\phi$ can be identified with the Sobolev space $H^{\tau}$, and
the corresponding norms are equivalent. 


Let $X=\{\vecx_1,\ldots,\vecx_N\}$ be a set of data points on the
sphere. Two important parameters characterising the set $X$ are the
{\em mesh norm} $h_X$ and {\it separation radius}
$q_X$, defined by
\begin{equation}\label{equ:mesh norm}
h_X 
:= 
\sup_{\vecy \in \mS} \min_{1\le j\le N} \cos^{-1}(\vecx_j\cdot\vecy)
\quad
\text{and}
\quad
q_X
:=
\frac 1 2\,
\min_{i\not= j\atop 1\le i,j \le N}\,
\cos^{-1}(\vecx_i\cdot\vecx_j).
\end{equation}
The {\em spherical radial basis functions} $\Phi_j$, $j=1,\ldots,N$,
associated with $X$ and the kernel $\Phi$ are defined by (see 
\eqref{equ:Phi res})
\begin{equation}\label{equ:Phi}
\Phi_j(\vecx) 
:= 
\Phi(\vecx,\vecx_j)
= 
\sum_{\ell=0}^\infty 
\sum_{m=1}^{N(n,\ell)} 
\widehat{\phi}(\ell)
{Y_{\ell,m}(\vecx_j)} Y_{\ell,m}(\vecx).
\end{equation}
We note that 
\begin{equation}\label{equ:Phi hat}
\wth{(\Phi_j)}_{\ell,m}
=
\widehat{\phi}(\ell) {Y_{\ell,m}(\vecx_j)},
\quad j=1,\ldots,N.
\end{equation}
It follows from \eqref{equ:con hat phi} that, for any $s\in\R$,
\[
\sum_{\ell=0}^\infty
\sum_{m=1}^{N(n,\ell)}
(\ell+1)^{2s}
\left|(\wth{\Phi_j})_{\ell,m}\right|^2
\simeq
\sum_{\ell=0}^\infty
\sum_{m=1}^{N(n,\ell)}
(\ell+1)^{2(s-2\tau)} |Y_{\ell,m}(\vecx_j)|^2.
\]
By using \eqref{equ:add for} and noting
$P_\ell(n;\vecx_j\cdot\vecx_j)=P_\ell(n;1)=1$ we obtain, recalling that
$N(n,\ell) =O(\ell^{n-2})$,
\[
\sum_{\ell=0}^\infty
\sum_{m=1}^{N(n,\ell)}
(\ell+1)^{2s}
\left|(\wth{\Phi_j})_{\ell,m}\right|^2
\simeq
\sum_{\ell=0}^\infty
(\ell+1)^{2(s-2\tau)+n-2}.
\]
The latter series converges if and only if $s<2\tau +(1-n)/2$.
Hence, 
\begin{equation}\label{equ:Phij Hs}
\Phi_j\in H^s
\quad
\Longleftrightarrow
\quad
s<2\tau + (1-n)/2.
\end{equation}

The finite-dimensional subspace to be used in our
approximation  is defined by
$\cV_X^\phi := \spann\{\Phi_1,\ldots,\Phi_N\}$.
This space is used by Kansa \cite{Kan90}
for collocation approximation. For brevity of notation 
we write $\cV^\phi$ for $\cV_X^\phi$ since there is no confusion.
Due to \eqref{equ:Phij Hs}, we have
\begin{equation}\label{equ:V Hs}
\cV^\phi\subset H^s
\quad \text{for all } s<2\tau+ \frac{1-n}{2}.
\end{equation}
We note that if $\tau > (n-1)/2$, then
$\cV^{\phi}\subset\cN_\phi\simeq H^{\tau}\subset C(\mS)$,
which is essentially the Sobolev embedding theorem.

It is noted that
if $\tau > (n-1)/2$ then 
any function $v\in\cN_\phi$ satisfies
\begin{equation}\label{equ:rep ker}
v(\vecx_j)
=
\sum_{\ell=0}^\infty \sum_{m=1}^{N(n,\ell)}
\frac{\widehat
v_{\ell,m}\widehat\phi(\ell)Y_{\ell,m}(\vecx_j)}{\widehat\phi(\ell)}
=
\inpro{v}{\Phi_j}_\phi,
\quad j=1,\ldots,N.
\end{equation}
This property is crucial in our analysis for the collocation method in 
Section~\ref{sec:Col Met}. 

We finish this subsection by proving the approximation property 
of $\cV^\phi$ as a subspace of Sobolev spaces.
This property is obtained by using the interpolation error 
which is derived  
in~\cite[Theorem~5.5]{NarSunWarWen07}.
This theorem states 
that if $v\in H^{s^*}$ for some $s^*$ satisfying 
$(n-1)/2< s^*\le \tau$ then for $0\le t^*\le s^*$ there holds
\begin{equation}\label{equ:Nar}
\norm{v - I_Xv}{t^*}
\leq 
C
\rho_X^{\tau - s^*}
h_X^{s^* - t^*}
\norm{v}{s}.
\end{equation}
Here, $\rho_X = h_X/q_X$, and $I_Xv\in\cV^\phi$ is the interpolant of $v$ at
$\vecx_j$,
$j=1,\ldots,N$, given by
\[
I_Xv(\vecx_j)
=
v(\vecx_j),
\quad
j=1,\ldots, N. 
\]
(In fact, it is required that $v\in\cN_{\phi}$ so that $I_Xv$
is well-defined.)
When solving pseudodifferential equations of order
$2\alpha$ by the Galerkin method, it is natural to carry out 
error analysis in the energy space $H^{\alpha}$. Since the
order $2\alpha$
may be negative (as in the case of the weakly-singular
integral equation discussed after Definition~\ref{def:L}) it
is necessary to show an approximation property of the form
\eqref{equ:Nar} for a wider range of
$t$ and $s$, including negative real values.

Before stating and proving the above mentioned approximation
property (Proposition~\ref{pro:app pro}), we recall the following property of
interpolation 
spaces which will be frequently used in 
the proof of that proposition.

\begin{lemma}\cite[Theorem B.2]{McL00}\label{lem:inter tech use}
Let $s_1, s_2,t_1,t_2\in \R$ be such that 
$s_1\le s_2$ and $t_1\le t_2$. Assume that 
$T:H^{s_i}\goto H^{t_i}$, $i=1,2$, are bounded linear 
operators satisfying 
\[
\norm{Tv}{t_i}
\le M_i
\norm{v}{s_i}
\quad
\forall
v\in H^{s_i},
\]
for some $M_i\ge 0$, $i=1,2$. Then for any $\theta\in[0,1]$,
$T:H^{\theta s_1+(1-\theta) s_2}\goto H^{\theta t_1+(1-\theta) t_2} $
is bounded and  there holds
\[
\norm{Tv}{\theta t_1+(1-\theta) t_2}
\le M_1^{\theta} M_2^{1-\theta}
\norm{v}{\theta s_1+(1-\theta) s_2}
\quad
\forall
v\in H^{\theta s_1+(1-\theta) s_2}.
\]

\end{lemma}

\begin{proposition}
\label{pro:app pro}
Assume that \eqref{equ:con hat phi} holds for some
$\tau>(n-1)/2$.
For any $s^*,t^*\in\R$ satisfying $t^* \le s^* \le 2\tau$ and
$t^*\le \tau$, 
if $v\in H^{s^*}$ then there exists $\eta\in \cV^{\phi}$ such
that
\begin{equation}\label{equ:app pro}
\norm{v-\eta}{t^*}
\le
Ch_X^{s^*-t^*}\norm{v}{{s^*}}
\end{equation}
for $h_X \le h_0$,
where $C$ and $h_0$ are independent of $v$ and $h_X$.
\end{proposition}
\begin{proof}
For $k = 0,1,2,\ldots$, we denote $\cI_k = [-k\tau, -(k-1)\tau]$ and
prove by  induction on $k$ that~\eqref{equ:app pro} holds for $t^*\in\cI_k$
for all $k$.

\noindent $\bullet$
We first prove that~\eqref{equ:app pro} is true when $t^*\in \cI_0$.
Indeed,  let $t^*\in \cI_0$. In this step, we consider 
two cases when $s^*$ belongs to $[\tau,2\tau]$ and $[t^*,\tau]$,
respectively.

\textbf{Case 1.1.} 
$\tau\le s^*\le 2\tau$.

Let $t$ and $s$ be real numbers satisfying 
$0\le t \le \tau \le
s \le 2\tau$. 
Let $I_Xv\in \cV^\phi$ be the interpolant of $v$ at
$\vecx_i$, $i=1,\ldots,N$. Then, by using \eqref{equ:rep
ker}, we deduce
\[
\inpro{v-I_Xv}{w}_\phi = 0
\quad\text{for all } w\in \cV^\phi.
\]
Hence, by using \eqref{equ:con hat phi} and the
Cauchy--Schwarz inequality, we obtain
for $v\in H^{2\tau}$
\begin{align}\label{equ:tau 2 tau}
\norm{v-I_Xv}{\tau}^2
&\simeq
\norm{v-I_Xv}{\phi}^2
=
\inpro{v-I_Xv}{v-I_Xv}_\phi \nonumber 
=
\inpro{v-I_Xv}{v}_\phi
\\
&\le
\sum_{\ell=0}^\infty \sum_{m=1}^{N(n,\ell)}
\frac{|\widehat v_{\ell,m} - \widehat{(I_Xv)}_{\ell,m}|
      |\widehat v_{\ell,m}|}{\widehat\phi(\ell)} \nonumber \\
&\simeq
\sum_{\ell=0}^\infty \sum_{m=1}^{N(n,\ell)}
(\ell+1)^{2\tau} |\widehat v_{\ell,m} - \widehat{(I_Xv)}_{\ell,m}|
      |\widehat v_{\ell,m}|
\le
\norm{v-I_Xv}{0} \norm{v}{2\tau}
\end{align}
Proposition 3.5 {in} \cite{TraLeGSloSte09a} gives
\begin{equation}\label{equ:v Iv 0 2tau}
\norm{v-I_Xv}{0}
\le
C h_X^{2\tau} \norm{v}{2\tau},
\end{equation}
which, together with \eqref{equ:tau 2 tau}, implies
\begin{equation}\label{equ:v Iv tau 2tau}
\norm{v-I_Xv}{\tau}
\le
C h_X^{\tau} \norm{v}{2\tau}.
\end{equation}
Noting the inequalities~\eqref{equ:v Iv 0 2tau}, \eqref{equ:v Iv tau 2tau},
and applying Lemma~\ref{lem:inter tech use} with $T = I-I_X$,
$s_1 = s_2 = 2\tau$, $t_1 = 0$, $t_2 = \tau$, and $\theta = (\tau-t)/\tau$,
we obtain
\begin{equation}\label{equ:v Iv t 2tau}
\norm{v-I_Xv}{t}
\le
C h_X^{2\tau-t} \norm{v}{2\tau},
\quad 0\le t \le \tau.
\end{equation}
On the other hand, by using \eqref{equ:Nar} with 
$t^*$ and $s^*$ replaced by $t$ and $\tau$, respectively,
we obtain
\begin{equation}\label{equ:v Iv t tau}
\norm{v-I_Xv}{t}
\le
C h_X^{\tau-t} \norm{v}{\tau},
\quad 0\le t \le \tau.
\end{equation}
Using Lemma~\ref{lem:inter tech use} again with $T = I - I_X$,
$t_1=t_2=t$, $s_1=\tau$, $s_2= 2\tau$, and $\theta = 2 - s/\tau$, we 
deduce 
\[
\norm{v-I_Xv}{t}
\le
C h_X^{s-t} \norm{v}{s},
\quad 0\le t \le \tau.
\]

Hence, we have proved 
\begin{equation}\label{equ:app inequ 0 t s 2tau}
\begin{cases}
0\le t^*\le\tau \le s^*\le 2\tau,
\\
\forall v\in H^{s^*},
\
\exists \eta_v = I_X v\in \cV^\phi:
\norm{v-\eta_v}{t^*}
\le
C h_X^{s^*-t^*} \norm{v}{s^*}.
\end{cases}
\end{equation}
\textbf{Case 1.2.} $t^* \le s^* <\tau$.

Let $s$ and $t$ be real numbers such that
$0\le s <\tau$ and  $2s-2\tau\le t\le s$. 
Let $P_{s} : H^{s}\to \cV^{\phi}$ be defined by
\begin{equation}\label{equ:Ps def v w 1}
\inpro{P_{s}v}{w}_{s}
=
\inpro{v}{w}_{s}
\quad\forall w\in \cV^{\phi}.
\end{equation}
It is easily seen that
\begin{equation}\label{equ:Ps 1}
\norm{v-P_{s}v}{s} \le \norm{v}{s}.
\end{equation}

If $2s-2\tau \le t\le 2s - \tau$ so that 
$\tau \le 2s-t \le 2\tau$ then
we apply~\eqref{equ:app inequ 0 t s 2tau} with $t^*$
and~$s^*$ replaced by $s$ and $2s-t$, respectively, to deduce 
that for any $w\in H^{2s-t}$, there exists $\eta_w\in\cV^\phi$ such that
\begin{equation}\label{equ:a p s 2s t 1}
\norm{w-\eta_w}{s} \le Ch_X^{s-t} \norm{w}{2s-t}.
\end{equation}
Since $\inprod{v- P_sv}{\eta_w}_s= 0$,
it follows from~\eqref{equ:neg nor}, \eqref{equ:CS s}, \eqref{equ:Ps 1}
and~\eqref{equ:a p s 2s t 1} that
\begin{align}
\norm{v-P_{s}v}{t}
&
=
\sup_{w\in H^{2s-t}\atop w\not=0}
\frac{\inpro{v-P_{s}v}{w}_{s}}{\norm{w}{2s-t}}
=
\sup_{w\in H^{2s-t}\atop w\not=0}
\frac{\inpro{v-P_{s}v}{w-\eta_w}_{s}}{\norm{w}{2s-t}}\notag \\
&
\le
\norm{v-P_{s}v}{s}
\sup_{w\in H^{2s-t}\atop w\not=0}
\frac{\norm{w-\eta_w}{s}}{\norm{w}{2s-t}}
\le
Ch_X^{s-t}\norm{v}{s}.\notag
\end{align}
In particular, for $t = 2s -\tau$ we have
\begin{equation}\label{equ:Ps s 2s s 1}
\norm{v-P_{s}v}{2s-\tau}
\le
Ch_X^{-s + \tau}\norm{v}{s}.
\end{equation}

If  $2s -\tau < t \le s$ then by 
noting \eqref{equ:Ps 1} and~\eqref{equ:Ps s 2s s 1}, and
applying Lemma~\ref{lem:inter tech use} with 
$T = I-P_s$, $s_1 = s_2 = s$, $t_1 = 2s-\tau$, 
$t_2 = s$, and $\theta = (t-s)/(s-\tau)$ we obtain $\norm{v-P_sv}{t}
\le Ch_X^{s-t}\norm{v}{s}$.

%
Combining both cases 1.1 and 1.2, we have proved that 
\begin{equation}\label{equ:app inequ 0 t s 2tau 2} 
\begin{cases}
t^*\in\cI_0,
\
t^*\le s^*\le 2\tau,
\\
\forall v\in H^{s^*}, \exists \eta_v\in\cV^\phi:
\norm{v-\eta_v}{t^*}
\le
C h_X^{s^*-t^*} \norm{v}{s^*}.
\end{cases}
\end{equation}

\noindent $\bullet$
Assume that for some $k_0\ge 0$, \eqref{equ:app pro} is true when $t^*\in
\cI_k$, for all $k=0,1,\ldots, k_0$, i.e., the following statement holds,

\begin{equation}\label{equ:statement k}
\begin{cases}
t^*\in \bigcup_{k=0}^{k_0}\cI_k,\ t^*\le s^*\le 2\tau,
\\
\forall v\in H^{s^*},\
\exists \eta_v\in \cV^\phi : 
\norm{v-\eta_v}{t^*}
\le
C h_X^{s^*-t^*}
\norm{v}{s^*}.
\end{cases}
\end{equation}

\noindent $\bullet$
We now prove that \eqref{equ:app pro} is also true when $t^*\in \cI_{k_0+1}$. 
Analogously to the case when $t^*\in\cI_0$,
we consider two cases when $s^*$ belongs to 
$[-k_0\tau, 2\tau]$ and $[t^*, -k_0\tau)$, respectively.

\textbf{Case 2.1.} $-k_0\tau \le s^*\le 2\tau$.

Let $t$ and $s$ be real numbers satisfying 
$t\in \cI_{k_0+1}$ and $s\in[-k_0\tau, 2\tau]$.
Let $P_{-k_0\tau} : H^{-k_0\tau}\to \cV^{\phi}$ 
be the projection defined by
\[
P_{-k_0\tau}v \in \cV^\phi: \quad
\inpro{P_{-k_0\tau}v}{w}_{-k_0\tau} = \inpro{v}{w}_{-k_0\tau}
\quad\forall w \in \cV^{\phi}.
\]
Then $P_{-k_0\tau}v$ is the best approximation of $v$ from $\cV^{\phi}$ in the
$H^{-k_0\tau}$-norm. It follows from~\eqref{equ:statement k} 
with ${-k_0\tau}$ and $s$ in place of $t^*$ and $s^*$, respectively, that
\begin{equation}\label{equ:P0 appro Pro}
\norm{v - P_{-k_0\tau}v}{-k_0\tau} \le Ch_X^{s + {k_0\tau}} \norm{v}{s}
\quad\forall v\in H^{s}.
\end{equation}
Since $t\in \cI_{k_0+1}$ so that $-k_0\tau\le -t-2k_0\tau\le 2\tau$, 
statement~\eqref{equ:statement k} with $t^*$ and $s^*$ replaced by 
$-k_0\tau $ and $-t-2k_0\tau$, respectively, assures that 
for any $w\in H^{-t-2k_0\tau}$, there exists  $\eta_w\in\cV^\phi$ 
such that 
\begin{equation}\label{equ:aus1}
\norm{w-\eta_w}{-k_0\tau}
\le 
C h_X^{-t-k_0\tau}
\norm{w}{-t-2k_0\tau}.
\end{equation}
Since $\inprod{v- P_{-k_0\tau}v}{\eta_w}_{-k_0\tau} = 0$, it follows 
from~\eqref{equ:neg nor} and~\eqref{equ:CS s} 
that
\begin{align*}
\norm{v - P_{-k_0\tau}v}{t}
&=
\sup_{w\in H^{-t - 2k_0\tau}\atop w\not=0}
\frac{\inpro{v-P_0v}{w}_{-k_0\tau}}{\norm{w}{{-t - {2k_0\tau}}}}
=
\sup_{w\in H^{-t - 2k_0\tau}\atop w\not=0}
\frac{\inpro{v-P_0v}{w- \eta_w}_{-k_0\tau}}{\norm{w}{{-t-2k_0\tau}}}\\
&\le
\norm{v-P_{-k_0\tau}v}{{-k_0\tau}}
\sup_{w\in H^{-t-2k_0\tau}\atop w\not=0}
\frac{\norm{w-\eta_w}{-k_0\tau}}{{\norm{w}{{-t-2k_0\tau}}}}. 
\end{align*}
Inequalities~\eqref{equ:P0 appro Pro} and~\eqref{equ:aus1} imply 
$\norm{v - P_{-k_0\tau}v}{t}\le Ch_X^{s-t}\norm{v}{s}$.

Hence, we have proved that 
\begin{equation}\label{equ:app inequ m 2tau t 0} 
\begin{cases}
-(k_0+1)\tau\le t^*\le -k_0\tau,
\
{-k_0\tau}\le s^*\le 2\tau,
\\
\forall v\in H^{s^*}, \exists \eta_v\in\cV^\phi:
\norm{v-\eta_v}{t^*}
\le
C h_X^{s^*-t^*} \norm{v}{s^*}.
\end{cases}
\end{equation}

\textbf{Case 2.2.} $t^* \le s^* < -k_0\tau$.

 Let $s$ and $t$ be real numbers such that 
$-(k_0+1)\tau \le s < -k_0\tau$ and
$2s-2\tau \le t \le s$.
Let $P_{s} : H^{s}\to \cV^{\phi}$ be defined by~\eqref{equ:Ps def v w 1} with
this new value of $s$. 

If $2s-2\tau \le t\le 2s + k_0\tau$ so that 
$-k_0\tau \le 2s-t \le 2\tau$ then
we can use the same argument as in Case~1.2 with~\eqref{equ:app inequ 0 t s
2tau} replaced by~\eqref{equ:app inequ m 2tau t 0} to obtain 
$\norm{v - P_sv}{t} \le C h_X^{s-t}\norm{v}{s}$.

If $2s+k_0\tau< t \le s$ then we use 
Lemma~\ref{lem:inter tech use} in the same manner as in Case~1.2 to obtain the
same estimate.

%
%

Combining both cases 2.1 and 2.2 we obtain the result for $k=k_0+1$, completing
the proof.
\end{proof}

\section{Approximate solutions}\label{sec:App sol}

\subsection{Approach}

Noting \eqref{equ:u u0 u1}, we 
shall seek an approximate
solution $\wtd u\in H^{\sigma+\al}$ in the form
\[
\wtd u = \wtd u_0 + \wtd u_1
\quad\text{where}\quad
\wtd u_0 \in \kerL
\quad\text{and}\quad
\wtd u_1 \in \cV^\phi.
\]
The solution $\wtd u_1$ will be found by the Galerkin or
collocation method.
Having found $\wtd u_1$, we will find $\wtd u_0\in \kerL$ 
by solving the equations (cf. \eqref{equ:sys equ})
\[
\inprod{\mu_i}{\wtd u_0}
=
\gamma_i - \inprod{\mu_i}{\wtd u_1},\quad i = 1,\ldots, M,
\]
so that
\begin{equation}\label{equ:mui wtd u}
\inprod{\mu_i}{\wtd u} = \inprod{\mu_i}u,
\quad i = 1,\ldots, M.
\end{equation}
The unique existence of $\wtd u_0$ follows from Assumption B in 
exactly the same way as that of $u_0$; see Proposition~\ref{the:ext uni}.

We postpone until Sections~\ref{sec:Gal} and~\ref{sec:Col Met} the issue of
finding $\wtd u_1$.
It is noted that in general
$\cV^\phi\notsubset(\kerL)_{H^{\sigma+\al}}^\perp$.
However, $\wtd u$ can be rewritten in a form similar to
\eqref{equ:u u0 u1} as follows. Let 
\begin{equation}\label{equ:u0s}
u_0^* 
:= 
\wtd u_0 
+ 
\sum_{\ell\in\cK(L)} \sum_{m=1}^{N(n,\ell)}
\wth{(\wtd u_1)}_{\ell,m} Y_{\ell,m}
\end{equation}
and
\begin{equation}\label{equ:u1s}
u_1^*
=
\sum_{\ell\notin\cK(L)} 
\sum_{m=1}^{N(n,\ell)} \wth{(\wtd u_1)}_{\ell,m} Y_{\ell,m}.
\end{equation}
Then 
\begin{equation}\label{equ:wtd u}
\wtd u = u_0^* + u_1^*
\quad\text{with}\quad
u_0^* \in \kerL
\quad\text{and}\quad
u_1^* \in (\kerL)_{H^{\sigma+\al}}^\perp.
\end{equation}
It should be noted that, in general, $u_1^*$ does not belong
to $\cV^\phi$, and that this function is introduced purely for
analysis purposes. We do not explicitly compute $u_1^*$, nor
$u_0^*$.

\subsection{Preliminary error analysis}

Assume that the exact solution $u$ and the approximate 
solution $\wtd u$ of Problem~A belong to $H^t$ for some 
$t\in\R$, and assume that  
$\mu_i\in H^{-t}$ for $i=1,\ldots,M$.
Comparing \eqref{equ:u u0 u1} and \eqref{equ:wtd u} suggests
that $\norm{u-\wtd u}{t}$ can be estimated by
estimating $\norm{u_0-u_0^*}{t}$ and
$\norm{u_1-u_1^*}{t}$. It turns out that an estimate for
the latter is sufficient, as shown in the
following two lemmas.
\begin{lemma}\label{lem:u0 u1}
Let $u_0$, $u_1$, $u_0^*$ and $u_1^*$ be defined by
\eqref{equ:u u0 u1}, \eqref{equ:u0s} and \eqref{equ:u1s}.
For $i=1,\ldots,M$,
if $\mu_i\in H^{-t}$ for some $t\in\R$, then
\[
\norm{u_0-u_0^*}{t}
\le
C \norm{u_1-u_1^*}{t},
\]
where $C$ is independent of $u$.
\end{lemma}
\begin{proof}
For $i=1,\ldots, M$,
it follows from \eqref{equ:mui wtd u} that
\[
\inprod{\mu_i}{u_0} + \inprod{\mu_i}{u_1}  
= 
\inprod{\mu_i}{u_0^*} +\inprod{\mu_i}{u_1^*},
\]
implying
$\inprod{\mu_i}{u_0-u_0^*}  
= 
\inprod{\mu_i}{u_1^*-u_1}$.
Inequality \eqref{equ:neg nor} with $s_1 = t$ and $s_2 = -t$ yields
\begin{align*}
\abs{\inprod{\mu_i}{u_0- u_0^*}}  
&
= 
\abs{\inprod{\mu_i}{u_1- u_1^*}}
\leq 
\norm{\mu_i}{-t}\norm{u_1- u_1^*}{t}.
\end{align*}
This result holds for all $i = 1,\ldots, M$, implying
\[
\norm{u_0-u_0^*}{\mu}\
\leq\  
\cM \,\norm{u_1-u_1^*}{t},
\]
where
$\mathcal M := \max_{i=1,\ldots, M}\norm{\mu_i}{-t}$, and
$\norm{v}\mu 
: =
\max_{i = 1,\ldots, M}\abs{\inprod{\mu_i}{v}}$
for all 
$v\in\kerL$.
(The unisolvency assumption assures us that the above norm is
well-defined.) The subspace $\kerL$ being finite-dimensional, we
deduce
\[
\norm{u_0-u_0^*}{t} 
\leq 
C\,\norm{u_1-u_1^*}{t},
\]
proving the lemma. 
\end{proof}

\begin{lemma}\label{lem:u u1}
Under the assumptions of Lemma~\ref{lem:u0 u1}, there holds
\[
\norm{u-\wtd u}{t}
\le
C \norm{u_1-u_1^*}{t}.
\]
\end{lemma}
\begin{proof}
Noting \eqref{equ:u u0 u1} and \eqref{equ:wtd u}, 
the norm $\norm{u-\wtd u}{t}$ can be rewritten as
\begin{align*}
\norm{u-\wtd u}{t}^2 
&= 
\sum_{\ell\in \mathcal K(L)}\,\sum_{m=1}^{N(n,\ell)}\,
(\ell+1)^{2t}\,
{|}\widehat{u}_{\ell,m} - \widehat{(\wtd u)}_{\ell,m}{|}^2  \\
&\quad+ 
\sum_{\ell\notin \mathcal K(L)}\,\sum_{m=1}^{N(n,\ell)}\,
(\ell+1)^{2t}\,
{|}\widehat{u}_{\ell,m} - \widehat{(\wtd u)}_{\ell,m}{|}^2
\\
&= 
\sum_{\ell\in \mathcal K(L)}\,\sum_{m=1}^{N(n,\ell)}\,
(\ell+1)^{2t}\,
{|}\widehat{(u_0)}_{\ell,m} - \widehat{(u_0^*)}_{\ell,m}{|}^2 \\
&\quad+ 
\sum_{\ell\notin \mathcal K(L)}\,\sum_{m=1}^{N(n,\ell)}\,
(\ell+1)^{2t}\,
{|}\widehat{(u_1)}_{\ell,m} - \widehat{(u_1^*)}_{\ell,m}{|}^2
\\
&= 
\norm{u_0 - u_0^*}{t}^2 +
\norm{u_1-u_1^*}{t}^2. 
\end{align*}
The required result now follows from Lemma~\ref{lem:u0 u1}.
\end{proof}

In the following sections, we describe methods
to construct $\wtd u_1$,
and estimate $\norm{u_1-u_1^*}{t}$ accordingly.

\section{Galerkin approximation}\label{sec:Gal}
Recalling \eqref{equ:V Hs}, we 
choose the shape functions $\phi$ 
in this subsection such that 
\begin{equation}\label{equ:tau al Gal}
\tau > \frac{1}{2}\left( \al + \frac{n-1}{2}\right),
\end{equation}
so that $\cV^\phi\subset H^{\al}$.
We find $\wtd u_1\in \cV^\phi$ by solving the Galerkin equation
\begin{equation}\label{equ:L wtd u1}
a(\wtd u_1,v) 
= 
\inprod{g}{v}\quad\text{for all }\,v\in \cV^\phi.
\end{equation}
By writing $\wtd u_1 = \sum_{i=1}^N c_i\Phi_i$ we derive 
from \eqref{equ:L wtd u1} the matrix equation
$
\vecA^{(G)}\vecc
=
\vecg,
$
where 
\begin{equation}\label{equ:matrix A}
\vecA^{(G)}_{ij}
=
a(\Phi_i,\Phi_j)
=
\sum_{\ell=0}^\infty
\sum_{m=1}^{N(n,\ell)} 
\wth L(\ell)  \, [\wth \phi(\ell)]^2 \, 
{Y_{\ell,m}(\vecx_i)} \, Y_{\ell,m}(\vecx_j),
\end{equation}
$\vecc = (c_1,\ldots, c_N)$, and 
$\vecg = \left(\inprod{g}{\Phi_1},\ldots,\inprod{g}{\Phi_N}\right)$.
\begin{lemma}\label{lem:A SG}
The matrix $\vecA^{(G)}$ is symmetric positive-definite.
\end{lemma}
\begin{proof}
Let $\theta$ be a shape function whose Fourier--Legendre 
coefficients are given by 
\[
\wth\theta(\ell)
=
\begin{cases}
\Lhe[\phe]^2 & \text{if}\ \ell\notin\cK(L)
\\
0            & \text{if}\ \ell\in\cK(L).
\end{cases}
\]
Then $\vecA^{(G)}_{ij} = \Theta(\vecx_i,\vecx_j)$ where 
$\Theta$ is the kernel defined from $\theta$. 
Since $\wth\theta(\ell) \ge 0$ for all $\ell\ge 0$, and
$\wth\theta(\ell)=0$ only for a finite number of $\ell$,
it follows from
Remark~\ref{rem:ker} that $\vecA^{(G)}$ is symmetric positive-definite.
\end{proof}
As a consequence of this lemma,
there exists a unique solution  $\wtd u_1$ 
to \eqref{equ:L wtd u1}.
With $\wtd u_1$ given by \eqref{equ:L wtd u1}, $u_1^*$
defined by \eqref{equ:u1s} satisfies $u_1^* \in H^\al$ and
\begin{equation}\label{equ:L u1s}
a(u_1^*,v)
=
\inpro{g}{v}
\quad\text{for all } v\in \cV^{\phi}.
\end{equation}
Even though in general $u_1^*$ does not belong to $\cV^\phi$, the
following result is essentially C\'ea's Lemma.
\begin{lemma}\label{lem:u1 u1s v}
If $u_1$ and $u_1^*$ are defined by \eqref{equ:Lu1}
and \eqref{equ:u1s} with $\wtd u_1$ given by 
\eqref{equ:L wtd u1}, then
\[
\norm{u_1-u_1^*}{\alpha}
\le
C  \norm{u_1-v}{\alpha}
\quad\text{for all } v\in \cV^\phi.
\]
\end{lemma}
\begin{proof}
It follows from the definition \eqref{equ:u1s} of $u_1^*$ that
\begin{equation}\label{equ:w u1}
a(w,u_1^*)
=
a(w,\wtd u_1)
\quad\text{for all } w\in H^\al.
\end{equation}
Moreover, since $\cV^\phi\subset H^\al\subset H^{\al-\sigma}$
(noting $\sigma\ge 0$) we infer from \eqref{equ:Lu1} and
\eqref{equ:L u1s}
\begin{equation}\label{equ:u1 u1s v}
a(u_1-u_1^*,v) = 0
\quad\text{for all } v\in \cV^\phi.
\end{equation}
Since $u_1 -u_1^*\in(\kerL)^\perp_{H^\al}$, Lemma~\ref{lem:a bil for}
yields
\[
\norm{u_1 - u_1^*}{\al}^2
\simeq
a(u_1-u_1^*, u_1-u_1^*)
=
a(u_1-u_1^*, u_1) - a(u_1-u_1^*, u_1^*).
\]
It follows from \eqref{equ:w u1} and \eqref{equ:u1 u1s v},
noting $u_1-u_1^*\in H^\al$ and $\wtd u_1\in \cV^\phi$, that
\[
\norm{u_1-u_1^*}{\al}^2
\simeq
a(u_1-u_1^*, u_1)
-
a(u_1-u_1^*, \wtd u_1)
=
a(u_1-u_1^*, u_1).
\]
Hence, using again \eqref{equ:u1 u1s v}, we obtain for any $v\in \cV^\phi$ 
\[
\norm{u_1-u_1^*}{\alpha}^2
\simeq
a(u_1-u_1^*,u_1-v)
\le
C \norm{u_1-u_1^*}{\alpha} \ \norm{u_1-v}{\alpha},
\]
where in the last step we used Lemma~\ref{lem:a bil for}.
By cancelling similar terms we obtain the required result.
\end{proof}

The above lemma and Proposition~\ref{pro:app pro} will be used to estimate the
error $u_1-u_1^*$.

\begin{lemma}\label{lem:hn1}
Assume that the shape function
$\phi$ is chosen to 
satisfy~\eqref{equ:con hat phi},~\eqref{equ:tau al Gal} and
$\tau\ge\al$,  $\tau>(n-1)/2$. 
Let $u_1$ and $u_1^*$ be defined as in Lemma~\ref{lem:u1 u1s v}. Assume that
$u_1\in H^s$ for some $s$ satisfying 
$\al\le s\le 2\tau$.
Let $t\in\R$ satisfy 
$2(\al-\tau)\le t \leq \al$. Then
for $h_X$ sufficiently small there holds
\begin{equation}\label{equ:brit 1}
\norm{u_1-u_1^*}{t}\
\leq\
Ch_X^{s-t}\,\norm{u_1}{s}.
\end{equation}
The constant $C$ is independent of $u$ and $h_X$.

\end{lemma}
\begin{proof}
The result for the case when $t=\al$ is a direct consequence of
Lemma~\ref{lem:u1 u1s v} and Proposition~\ref{pro:app pro} (with $t^* = \alpha$
and $s^* = s$).

The proof for the case $t<\al$ is standard, using Aubin--Nitsche's trick,
and is included here for completeness. 
It follows from \eqref{equ:neg nor} and 
\eqref{equ:strongly ellip} that
\[
\norm{u_1-u_1^*}{t}
\leq 
\sup_{{v\in H^{2\al-t} \atop v \not= 0}}
\frac{\inpro{u_1- u_1^*}{v}_\al}{\norm{v}{2\al-t}}
\le
C \sup_{{v\in H^{2\al-t} \atop v \not= 0}}
\frac{a(u_1- u_1^*,v)}{\norm{v}{2\al-t}}.
\]
By using successively \eqref{equ:u1 u1s v}, Lemma~\ref{lem:a bil for},
\eqref{equ:brit 1} with $t$ replaced by $\alpha$, and~\eqref{equ:ame 1},
we deduce for any $\eta\in\cV^{\phi}$
\begin{align}\label{equ:hn2}
\norm{u_1-u_1^*}{t}
&
\le
C \sup_{{v\in H^{2\al-t} \atop v \not= 0}}
\frac{a(u_1- u_1^*, v-\eta)}{\norm{v}{2\al-t}}
\le
C \norm{u_1- u_1^*}{\al}
\sup_{{v\in H^{2\al-t} \atop v \not= 0}}
\frac{\norm{v-\eta}{\al}}{\norm{v}{2\al-t}}
\nonumber \\
&
\le 
C h_X^{s-\alpha}\norm{u_1}{s}
\sup_{{v\in H^{2\al-t} \atop v \not= 0}}
\frac{\norm{v-\eta}{\al}}{\norm{v}{2\al-t}}.
\end{align}
Since $2(\al-\tau) \le t < \al$, there holds $\al < 2\al-t\le2\tau$.
By invoking Proposition~\ref{pro:app pro} again with $t^*$ and $s^*$ replaced
by $\alpha$ and $2\alpha-t$, respectively, we can choose
$\eta\in \cV^\phi$ satisfying
\begin{equation}\label{equ:ame 1}
\norm{v-\eta}{\al}
\le
C h_X^{\al-t} \norm{v}{2\al-t}.
\end{equation}
This together with \eqref{equ:hn2} yields the required
estimate, proving the lemma.
\end{proof}

We are now ready to state and prove the main result of this section.
\begin{theorem}\label{the:Gal}
Assume that the shape function
$\phi$ is chosen to 
satisfy~\eqref{equ:con hat phi},~\eqref{equ:tau al Gal} and
$\tau\ge\al$,  $\tau>(n-1)/2$. 
Assume further that
$u\in H^s$ for some $s$ satisfying 
$\al\le s\le 2\tau$.
If $\mu_i\in H^{-t}$ for $i=1,\ldots, M$ with
$t\in\R$ satisfying
$2(\al-\tau)\le t \leq \al$, then
for $h_X$ sufficiently small there holds
\begin{equation*}
\norm{u-\wtd u}{t}\
\leq\
Ch_X^{s-t}\,\norm{u}{s}.
\end{equation*}
The constant $C$ is independent of $u$ and $h_X$.
\end{theorem}

\begin{proof}
Since $\mu_i\in H^{-t}$ for $i=1,\ldots, M$, Lemma~\ref{lem:u u1}
gives
\[
\norm{u-\wtd u}{t}
\le 
C
\norm{u_1-u_1^*}{t}.
\]
The required result is a consequence of Lemma~\ref{lem:hn1},
noting that $\norm{u_1}{s}\le\norm{u}{s}$.
\end{proof}

\section{Collocation approximation}\label{sec:Col Met}

Recall that for this method it is assumed that $g\in H^{\sigma-\alpha}$ for some
positive $\sigma$
so that $u\in H^{\sigma+\alpha}$; see Problem~A. 
We will assume that 
\begin{equation}\label{equ:sig tau col con}
\max\{2\alpha, \alpha\} +\frac{n-1}{2}
< \tau \le \min\{\sigma-\alpha,\sigma\}.
\end{equation}
Recall that~\eqref{equ:con hat phi} implies $\cN_\phi\simeq H^\tau$.
Thus, the condition $\sigma-\alpha\ge \tau$ assures us that $g\in
\cN_\phi$. The condition $2\alpha +(n-1)/2 < \tau$ is to assure that $L\wtd
u_1\in\cN_\phi$. Indeed, this condition implies $\wtd u_1\in\cV^\phi\subset
H^{\tau+2\alpha}$ which is equivalent to $L\wtd u_1\in H^\tau\simeq
\cN_\phi$.

The functions $L\wtd u_1$ and $g$ are required to be in the native space
$\cN_\phi$ so that property~\eqref{equ:rep ker} can be used. The 
conditions $\alpha+(n-1)/2 \le \tau$
and $\tau\le\sigma$ are purely technical requirements of our proof.

In this method we find
$\wtd u_1 \in \cV^\phi$ 
by solving the collocation equation
\begin{equation}\label{equ:col equ}
L\wtd u_1(\vecx_j)
=
g(\vecx_j),
\quad
j=1,\ldots,N.
\end{equation}
By writting $\wtd u_1 = \sum_{j=1}^N c_j\Phi_j$, we derive from~\eqref{equ:col
equ} the matrix equation $\vecA^{(C)} \vecc =\vecg$ where 
%
\[
\vecA^{(C)}_{ij}
=
L\Phi_i(\vecx_j)
=
\sum_{\ell=0}^\infty
\sum_{m=1}^{N(n,\ell)} 
\wth L(\ell)  \, \wth \phi(\ell) \, 
{Y_{\ell,m}(\vecx_i)} \, Y_{\ell,m}(\vecx_j),
\]
$\vecc = (c_1,\ldots, c_N)$ and $\vecg = (g(\vecx_1, \ldots, g(\vecx_N))$.
The symmetry and positive definiteness of 
the matrix $\vecA^{(C)}$  can be proved in the same manner as 
Lemma~\ref{lem:A SG}.

Since the function $\Phi$ defined  as in
\eqref{equ:Phi res} is a reproducing kernel for the Hilbert
space~$\cN_\phi$, see \eqref{equ:rep ker}, the collocation
equation \eqref{equ:col equ} can be rewritten as a Galerkin
equation. This allows us to carry out error analysis in the
same manner as in Section~\ref{sec:Gal}.

Recalling \eqref{equ:rep ker} and noting that $L\wtd u_1, g\in \cN_\phi$, we
rewrite \eqref{equ:col
equ} as
\begin{equation}\label{equ:u1 phi}
\inprod{L\wtd u_1}{\Phi_j}_\phi
=
\inprod{g}{\Phi_j}_\phi,
\quad
j=1,\ldots, N.
\end{equation}
In order to see that the above equation is a Galerkin
equation, we introduce a new finite-dimensional subspace
$\cV^{\wtd{\phi}}$ :
\[
\cV^{\wtd{\phi}}
:=
\spann\{\wtd{\Phi}_1,\ldots,\wtd{\Phi}_N\},
\]
where the spherical radial basis functions $\wtd \Phi_j$ are defined 
by
\[
\wtd{\Phi}_j(\vecx)
:=
\wtd{\phi}(\vecx\cdot\vecx_j),
\quad j=1,\ldots,N.
\]
Here, $\wtd\phi$ is a shape function given by
\[
\wtd{\phi}(t)
:=
\sum_{\ell=0}^\infty 
\omega_n^{-1} N(n,\ell)
\big[\wth\phi(\ell)\big]^{1/2}
P_\ell(n;t),
\]
It is easily seen that (cf. \eqref{equ:Phi hat})
\begin{equation}\label{equ:Psi hat}
\wth{(\wtd{\Phi}_j)}_{\ell,m}
=
[\wth\phi(\ell)]^{1/2} \,
{Y_{\ell,m}(\vecx_j)},
\quad j=1,\ldots,N.
\end{equation}
It should be noted that this space $\cV^{\wtd\phi}$ is
introduced purely for analysis purposes; it is
 not to be used in the
implementation.
Since (cf. \eqref{equ:con hat phi})
\[
c_1 (\ell+1)^{-\tau}
\le
\wth{(\wtd{\phi})}(\ell)
\le
c_2 (\ell+1)^{-\tau},
\]
we have (cf. \eqref{equ:V Hs})
\begin{equation}\label{equ:V phi tld H}
\cV^{\wtd{\phi}} \subset H^s
\quad\text{for all } s < \tau+\frac{1-n}{2}.
\end{equation}
In particular, $\cV^{\wtd{\phi}}\subset H^{\al}$ due to 
$\alpha+(n-1)/2 < \tau$ (see~\eqref{equ:sig tau col con}).

The following lemma defines a weak equation 
equivalent to equation~\eqref{equ:Lu1}.

\begin{lemma}\label{lem:U1}
Let
\begin{equation}\label{equ:U1 Col Pet}
U_1
:=
\sum_{\ell\notin\cK(L)}
\sum_{m=1}^{N(n,\ell)}
\frac{(\wth u_1)_{\ell,m}}{\big[\wth\phi(\ell)\big]^{1/2}}
Y_{\ell,m},
\end{equation}
where $u_1$ is the solution to \eqref{equ:Lu1}.
Then $U_1$ belongs to $H^{\sigma+\al-\tau}$ and satisfies
\begin{equation}\label{equ:LU1}
a(U_1,V)
=
\inpro{G}{V}
\quad\text{for all }\, V\in H^{\al-\sigma+\tau},
\end{equation}
where
\begin{equation}\label{equ:G def g}
G
:=
\sum_{\ell=0}^\infty
\sum_{m=1}^{N(n,\ell)}
\frac{\wth g_{\ell,m}}{\big[\wth\phi(\ell)\big]^{1/2}}
Y_{\ell,m}.
\end{equation}
\end{lemma}
\begin{proof}
Since $u_1\in H^{\sigma+\al}$, it is easily seen that
$U_1\in H^{\sigma+\al-\tau}$.
For any $V\in H^{\al-\sigma+\tau}$ there holds
\[
a(U_1,V)
=
a(u_1,v),
\]
where
\[
v
:=
\sum_{\ell=0}^\infty
\sum_{m=1}^{N(n,\ell)}
\frac{\wth{V}_{\ell,m}}{\big[\wth\phi(\ell)\big]^{1/2}}
Y_{\ell,m}.
\]
Noting $v\in H^{\al-\sigma}$ we deduce from \eqref{equ:Lu1}
that
\[
a(U_1,V)
=
\inpro{g}{v}
=
\inpro{G}{V},
\]
finishing the proof of the lemma. 
\end{proof}

Analogously, the next lemma defines an equivalent to~\eqref{equ:u1 phi}. It
will be seen later that this equivalent is the Galerkin approximation
to~\eqref{equ:LU1}.

\begin{lemma}\label{lem:wtd U1}
Let
\begin{equation}\label{equ:U1t col}
\wtd U_1
:=
\sum_{\ell=0}^\infty
\sum_{m=1}^{N(n,\ell)}
\frac{\wth{(\wtd
u_1)}_{\ell,m}}{\big[\wth\phi(\ell)\big]^{1/2}}
Y_{\ell,m}
\end{equation}
where $\wtd u_1$ is given by \eqref{equ:col equ}.
Then $\wtd U_1$ belongs to $\cV^{\wtd{\phi}}$ and satisfies
\begin{equation}\label{equ:LU1 wtd}
a(\wtd U_1,\wtd{\Phi}_j)
=
\inprod{G}{\wtd{\Phi}_j},
\quad
j=1,\ldots, N.
\end{equation}
\end{lemma}
\begin{proof}
Since $\wtd u_1\in \cV^\phi$ we have
$\wtd u_1
=
\sum_{j=1}^N
c_j \Phi_j$ for some $c_j \in \R$,
which together with \eqref{equ:Phi hat} implies
\[
\wth{(\wtd u_1)}_{\ell,m}
=
\wth\phi(\ell)
\sum_{j=1}^N
c_j {Y_{\ell,m}(\vecx_j)}.
\]
This in turn gives
\[
\wth{(\wtd U_1)}_{\ell,m}
=
[\wth\phi(\ell)]^{1/2}
\sum_{j=1}^N
c_j {Y_{\ell,m}(\vecx_j)},
\]
so that (see \eqref{equ:Psi hat})
\[
\wtd U_1
=
\sum_{j=1}^N
c_j \wtd{\Phi}_j,
\]
i.e., $\wtd U_1 \in \cV^{\wtd{\phi}}$.
By using successively \eqref{equ:bilinear srbf}, \eqref{equ:Psi hat},
\eqref{equ:U1t col}, \eqref{equ:u1 phi}, \eqref{equ:Phi hat} and~\eqref{equ:G
def g}, we deduce
\[
a(\wtd U_1,\wtd{\Phi}_j)
=
\inpro{L\wtd U_1}{\wtd{\Phi}_j}
=
\inpro{L\wtd u_1}{\Phi_j}_\phi
=
\inpro{g}{\Phi_j}_\phi
=
\inpro{G}{\wtd{\Phi}_j},
\quad j=1,\ldots,N,
\]
completing the proof of the lemma. 
\end{proof}

Using the two above lemmas we can now estimate the error in the collocation
approximation in the same manner as for the Galerkin approximation.

\begin{theorem}\label{the:str col}
Let 
\eqref{equ:sig tau col con} hold.
We choose the shape function $\phi$
such that~\eqref{equ:con hat phi} holds with
$\tau > n-1$. 
Assume further that $u\in H^s$ for some $s$ satisfying 
$\tau+\al\leq s\leq 2\tau$.
If $\mu_i\in H^{-t}$, $i=1,\ldots, M$ for some
$t$ satisfying $2\al\leq t\leq \tau+\al$, then
for $h_X$ sufficiently small there holds
\[
\norm{u-\wtd u}{t}
\leq 
C
h_X^{s-t}
\norm{u}{s}.
\]
The constant $C$ is independent 
of $u$ and $h_X$.
\end{theorem}

\begin{proof} 
Recall that $\wtd U_1\in\cV^{\wtd\phi}\subset H^{\alpha}$ and $U_1\in
H^{\sigma+\alpha-\tau}\subset H^{\alpha}$ since $\tau\le \sigma$;
see~\eqref{equ:sig tau col con}. 
Moreover,~\eqref{equ:LU1} and~\eqref{equ:LU1 wtd} imply
\[
a(U_1-\wtd U_1, \wtd \Phi_j)
=
0,
\quad
j=1,\ldots,N.
\]
Hence, $\wtd U_1\in\cV^{\wtd\phi}$ is the Galerkin 
approximation to $U_1$. 

Analogously to~\eqref{equ:u1s}
we define 
\begin{equation}\label{equ:brit 2}
U_1^* = \slnotinK \wth{(\wtd U_1)}\Ylm.
\end{equation}
Lemma~\ref{lem:hn1}
with $\cV^\phi$ replaced by $\cV^{\wtd\phi}$ (and therefore,
$\tau$ replaced by $\wtd\tau:=\tau/2$) and $u_1, u_1^*$ replaced by $U_1,
U_1^*$,
 gives
\begin{equation}\label{equ:brit 3}
\norm{U_1- U_1^*}{\wtd t}
\leq
C
h_X^{\wtd s- \wtd t}
\norm{U_1}{\wtd s},
\quad
\alpha\le \wtd s\le 2\wtd \tau,\
2(\alpha-\wtd\tau)\le \wtd t\le \alpha.
\end{equation}
By the definition of $U_1, \wtd U_1$ and $U_1^*$, see~\eqref{equ:U1 Col Pet},
\eqref{equ:U1t col} and~\eqref{equ:brit 2}, we have 
\begin{equation}\label{equ:brit 4}
\norm{u_1 - u_1^*}{t}
\simeq 
\norm{U_1 - U_1^*}{t-\tau}
\
\text{and}
\
\norm{u_1}{s}
\simeq
\norm{U_1}{s - \tau}.
\end{equation}
Since $t$ and $s$ satisfy $2\alpha\le t\le \tau+\alpha$ and $\tau+\alpha\le s\le
2\tau$ so that $t-\tau$ and  $s-\tau$ satisfy
\[
2(\alpha-\wtd\tau)
\le
t-\tau\le \alpha
\quad\text{and}\quad
\alpha\le s-\tau\le 2\wtd\tau,
\]
the inequality~\eqref{equ:brit 3} with $\wtd t =t-\tau$ and $\wtd
s=s-\tau$ gives
\[
\norm{U_1- U_1^*}{t -\tau}
\leq
C
h_X^{ s-  t}
\norm{U_1}{s-\tau}.
\]
This together with~\eqref{equ:brit 4} implies
\[
\norm{u_1 - u_1^*}{t}
\leq
C
h_X^{ s-  t}
\norm{u_1}{s} 
\]
Since $\mu_i\in H^{-t}$, for $i=1,\ldots, M$, by using Lemma~\ref{lem:u u1} and
noting that $\norm{u_1}{s}
\le \norm{u}{s}$, we deduce
\[
\norm{u-\wtd u}{t}
\le
C
\norm{u_1 - u_1^*}{t}
\leq
C
h_X^{ s-  t}
\norm{u_1}{s} 
\le 
C
h_X^{ s-  t}
\norm{u}{s}, 
\]
completing the proof of the theorem.

\end{proof}

\begin{remark}
{\rm
In comparison with the results obtained 
by Morton and Neamtu, our error estimates 
for the collocation approximation
cover a wider range of Sobolev norms. In fact, these two 
authors only proved~\cite{Mor00}
\[
\norm{u-\wtd u}{2\alpha}
\le
ch_X^{\lfloor 2(\tau-\al) \rfloor}
\norm{u}{2\tau}.
\]
This is a special case of the results in
Theorems~\ref{the:str col}.
}
\end{remark}

\section{Numerical experiments}\label{sec:num exp srbf}

In this section, we solved the Dirichlet problem
\begin{equation}\label{equ:Dirichlet problem}
\begin{aligned}
\Delta U 
&
= 0 \ \text{in}\ \mathbb B_e,
\\
U 
&
= U_D \ \text{on}\ \mS,
\\
U(\vecx) 
&
= O(1/\snorm{\vecx}{})\ \text{as}\ \snorm{\vecx}{}\goto\infty,
\end{aligned}
\end{equation}
where $\mathbb B_e:=\{\vecx\in \R^3 : \snorm{\vecx}{}> 1\}$. 
It is well-known, see e.g.~\cite{PhaTraLeG08}, that the
problem~\eqref{equ:Dirichlet problem}
is equivalent to  
\begin{equation}\label{equ:Su f equation}
Su  = g \ \text{on}\ \mS,
\end{equation}
where
\begin{equation}\label{equ:g def 1s}
g = -\frac{1}{2} U_D + DU_D,
\end{equation}
and
\[
Dv(\vecx)
=
\frac{1}{4\pi}
\int_{\mS}
v(\vecy)
\frac{\partial}{\partial\nu_{\vecy}}
\frac{1}{\snorm{\vecx-\vecy}{}}\,d\sigma_{\vecy}.
\]
Here, $S$ is the weakly singular integral operator 
defined by
\[
Sv(\vecx)
=
\frac{1}{4\pi}
\int_{\mSt}
\frac{v(\vecy)}{\snorm{\vecx-\vecy}{}}\,d\sigma_{\vecy},
\]
which is a
pseudodifferential operator of order $-1$ and $\wth S(\ell) =
1/(2\ell+1)$; see the examples following Definition~\ref{def:L}.

We solved the problem~\eqref{equ:Dirichlet problem} with the boundary data
\[
U_D(\vecx)
:=
U_D(x_1,x_2,x_3) 
=
\frac{1}{(1.0625 - 0.5x_3)^{1/2}}
\]
so that the exact solution to the Dirichlet
problem~\eqref{equ:Dirichlet problem} is given by 
\[
U(\vecx) 
=
\frac{1}{\snorm{\vecx - \vecq}{×}}
\quad \text{with}\quad \vecq = (0,0,0.25),
\]
and hence, the exact solution to the weakly singular integral
equation~\eqref{equ:Su f equation} is
$u(\vecx)
=
\partial_{\nu} U(\vecx) 
$; see e.g.~\cite{PhaTraLeG08}, i.e.,
\[
u(\vecx)
=
\frac{-1 + \vecx\cdot\vecq}{\snorm{\vecx -\vecq}{×}^3}
=
\frac{0.25 x_3 - 1}{(1.0625 - 0.5x_3)^{3/2}}.
\]
For the approximation of~\eqref{equ:Su f equation}, we use spherical radial
basis functions suggested by Wendland~\cite[page 128]{Wen95}. 
The sets $X:=\{\vecx_1,\vecx_2,\ldots,\vecx_N\}$ of points are chosen purely to
observe
the order of convergence. Experiments with real data can be found
in~\cite{PhaTraLeG08}.

The shape function $\phi:[-1,1]\goto\R$ which is used to define the kernel
$\Phi$ is given by
\begin{equation}\label{equ:rho m}
\phi(t)=\rho(\sqrt{2-2t}), 
\end{equation}
where $\rho$ is 
Wendland's functions \cite[page 128]{Wen05} defined
by
\[
\rho(r)=
(1-r)^2_{+}. 
\]
%
Narcowich and Ward \cite[Proposition 4.6]{NarWar02} prove that $\phe\sim
(1+\ell)^{-2\tau}$ for all $\ell\ge 0$, where $\tau = 3/2$. 
The spherical radial basis
functions $\Phi_i$, $i=1,\ldots,N$, are computed by
\begin{equation}\label{equ:Phi i def rho}
\Phi_i(\vecx)
=
\rho(\sqrt{2-2\vecx\cdot\vecx_i}),
\quad\vecx\in\mS.
\end{equation}

We first found an approximate solution 
$u_X^G\in\cV_X^\phi:= \spann\{\Phi_1,\Phi_2,\ldots,\Phi_N \}$ 
satisfying the Galerkin
equation
\begin{equation}\label{equ:weakly Galerkin equation}
a_S(u_X^G, v)
:=
\inprod{Su_X^G}{v}
=
\inprod{g}{v}
\quad
\forall
v\in\cV_X^\phi. 
\end{equation}
%
The stiffness matrix arising from~\eqref{equ:weakly Galerkin equation}
has entries given by 
\begin{align*}
a_S(\Phi_i,\Phi_j)
&
=
\sum_{\ell=0}^\infty
\frac{|{\phe}{}|^2}{2\ell+1}
\sum_{m = -\ell}^\ell
\Ylm(\vecx_i)
\Ylm(\vecx_j)
=
\frac{1}{4\pi}
\sum_{\ell=0}^\infty
|{\phe}{}|^2
P_\ell(\vecx_i\cdot\vecx_j).
\end{align*}
The right-hand side of~\eqref{equ:weakly Galerkin equation} is computed by 
using~\eqref{equ:g def 1s}, noting $\wth{D}(\ell) = -1/(4\ell+2)$
(see~\cite[page
122]{Ned00}), 
\begin{align*}
\inprod{g}{\Phi_i}
&
=
\sum_{\ell=0}^\infty
\sum_{m = -\ell}^\ell
\Big{(}
-\frac{1}{2} - \frac{1}{2(2\ell+1)}
\Big{)}
\wth{(U_D)}_{\ell,m} \phe \Ylm(\vecx_i)
\\
&
=
-
\sum_{\ell=0}^\infty
\sum_{m = -\ell}^\ell
\frac{(\ell+1)}{2\ell+1}
\wth{(U_D)}_{\ell,m}
\phe \Ylm(\vecx_i).
\end{align*}
The errors are computed by 
\begin{equation}\label{equ:Hminus1o2 error compute}
\norm{u - u_X^G}{-1/2}
=
\left(
\sum_{\ell=0}^{\infty}
\sum_{m=-\ell}^\ell
\frac{\snorm{\uhlm - \wth{(u_X^G)}_{\ell,m}}{×}^2}{\ell+1}
\right)^{1/2}.
\end{equation}
 
Our theoretical result (Theorem~\ref{the:Gal}) predicts an order of convergence
of $2\tau+1/2$ in the $H^{-1/2}$-norm. We carried out the
experiment and observed some agreement between the experimented orders of
convergence (EOC) and our theoretical results; see Tables~\ref{tab:Errors
Hminus12 m
0}.
 
\begin{table}[ht]
\caption{Galerkin method: Errors in $H^{-1/2}$-norm, $\tau = 1.5$.
Expected order of convergence : 3.5}
\begin{center}
\begin{tabular}{|c|c|c|c|}
\hline
       N   &    $h_X$      & $H^{-1/2}$-norm &      EOC       \\
\hline       
      20   &   0.65140   &   0.120349381   &           \\
      30   &   0.51210   &   0.054895875   &     3.262      \\
      40   &   0.44180   &   0.025612135   &     5.163      \\
      51   &   0.37500   &   0.015883257   &     2.915      \\
     101   &   0.26720   &   0.006082010   &     2.832      \\
     200   &   0.19420   &   0.001977985   &     3.520      \\
     500   &   0.12370   &   0.000492078   &     3.084      \\
\hline
\end{tabular}
\end{center}
\label{tab:Errors Hminus12 m 0}
\end{table}

The collocation solution $u_X^C\in \cV_X^\phi$ is 
found by solving 
\begin{equation}\label{equ:weakly collocation equation}
Su_X^C(\vecx_i)
=
g(\vecx_i),
\quad
i=1,\ldots, N.
\end{equation}
By writing $u_X^C = \sum_{i=1}^Nc_i\Phi_i$, we derive 
from~\eqref{equ:weakly collocation equation} the matrix
equation
\[
\vecS^C \vecc 
=
\vecg,
\]
where $\vecc = (c_i)_{i=1,\ldots, N}$, $\vecg = (g(\vecx_i))_{i=1,\ldots, N}$
and 
\[
\vecS^C_{ij}
=
S\Phi_i(\vecx_j) 
=
\sum_{\ell=0}^\infty
\sum_{m=-\ell}^\ell
 \frac{\phe}{2\ell+1}\Ylm(\vecx_i)\Ylm(\vecx_j),
\quad i,j =1,\ldots, N.
\]
By using the addition formula~\eqref{equ:add for}, we obtain 
\begin{equation}\label{equ:Aij Weakly}
\vecS_{ij}^C
=
\frac{1}{4\pi}
\sum_{\ell=0}^\infty
\phe
P_{\ell}(\vecx_i\cdot\vecx_j).
\end{equation}

The errors are then computed similarly as in~\eqref{equ:Hminus1o2 error
compute}. 
There is agreement between the experimented order of convergence
(EOC) and our theoretical result (which is $2\tau+1/2$);
 see Tables~\ref{tab:Errors
Hminus12 m 0 collocation}.

\begin{table}[ht]
\caption{Collocation method: Errors in $H^{-1/2}$-norm, $\tau = 1.5$.
Expected order of convergence : 3.5}
\begin{center}
\begin{tabular}{|c|c|c|c|}
\hline
       N   &    $h_X$      & $H^{-1/2}$-norm &      EOC       \\
\hline       
      20   &   0.65140   &   0.139479793   &               \\
      30   &   0.51210   &   0.047806025   &     4.450      \\
      40   &   0.44180   &   0.020666895   &     5.679      \\
      51   &   0.37500   &   0.011785692   &     3.426      \\
     101   &   0.26720   &   0.003674365   &     3.439      \\
     400   &   0.12370   &   0.000277996   &     3.352      \\
\hline
\end{tabular}
\end{center}
\label{tab:Errors Hminus12 m 0 collocation}
\end{table}

\section*{Acknowledgements}
The authors would like to thank Dr. Thong Le Gia for providing 
some parts of the code which are used in the numerical experiments.
The first author is supported by the University International
Postgraduate Award offered by the University of New South
Wales.
The second author is partially supported by the grant FRG PS17166.



\end{document}